\documentclass[12pt]{amsart}
\usepackage[active]{srcltx}
\usepackage{a4wide}
\usepackage{amsthm,amsfonts,amsmath,mathrsfs,amssymb}
\usepackage{dsfont}
\usepackage[T1]{fontenc}
\usepackage[utf8]{inputenc}
\usepackage{fourier}
\usepackage{pdfpages}

       \def\la{\lambda}     
              
         \def\r{\rho}

\def\e{\varepsilon}

\def\F{\Phi}

\def\D{{\mathbb D}}

   \def\cb{{\mathcal B}}
   \def\cd{{\mathcal D}}

\def\({\left(}       \def\){\right)}

\newcommand\norm[1]{\left\lVert#1\right\rVert}
\newtheorem{prop}{\sc Proposition}
\newtheorem{lem}[prop]{\sc Lemma}
\newtheorem{thm}[prop]{\sc Theorem}

\newtheorem{ex}[prop]{\sc Example}

\begin{document}
\title[Estimates for truncated area functionals]{Estimates for truncated area functionals on the Bloch space}
\author[I. Efraimidis]{Iason Efraimidis}
\address{Departamento de Matem\'aticas, Universidad Aut\'onoma de
Madrid, 28049 Madrid, Spain}
\email{iason.efraimidis@uam.es}
\author[A. Mas]{Alejandro Mas}
\address{Departamento de Análisis Matemático, Facultad de Ciencias, 
Universidad de Málaga, 29071 Málaga, Spain}
\email{alejandro.mas@uma.es}
\author[D. Vukoti\'c]{Dragan Vukoti\'c}
\address{Departamento de Matem\'aticas, Universidad Aut\'onoma de
Madrid, 28049 Madrid, Spain} \email{dragan.vukotic@uam.es}
\thanks{All authors are partially supported by PID2019-106870GB-I00 from MICINN, Spain. The first author is supported by a María Zambrano contract, reference number CA3/RSUE/2021-00386, from UAM and Ministerio de Universidades, Spain (Plan de Recuperación, Transformación y Resiliencia).}
\date{08 December, 2022.}

\begin{abstract}
Recently, Kayumov \cite{K} obtained a sharp estimate for the $n$-th truncated area functional for normalized functions in the Bloch space for $n\le 5$ and then, together with Wirths \cite{KW1}, extended the result for $n=6$. We prove that for the functions with non-negative Taylor coefficients, the same sharp estimate is valid for all $n$. For arbitrary functions, we obtain an estimate that is asymptotically of the same order but slightly larger (roughly by a factor of $4/e$). We also consider related weighted estimates for functionals involving the powers $n^t$, $t>0$, and show that the  exponent $t=1$ represents the critical case for the expected sharp estimate.
\end{abstract}
\maketitle
\section{Introduction}
 \label{sect-intro}
Let $\D$ denote the unit disc in the complex plane. The \textit{Bloch space\/} $\cb$ (see \cite{ACP}) is the Banach space of all analytic functions in $\D$ with the finite Bloch norm:
$$
 \|f\|_\cb = |f(0)| + \sup_{z\in\D} (1-|z|^2) |f^\prime(z)| \,.
$$
For a function in $\cb$ that vanishes at the origin, written as a power series $f(z)=\sum_{n=1}^\infty b_n z^n$, the following sharp estimate for the $n$-th Taylor coefficient is known \cite{W}:
$$
 B_n = \sup\{|b_n|\,\colon\,\|f\|_\cb\le 1\} = \frac{n+1}{2n} \( \frac{n+1}{n-1} \)^{(n-1)/2}\,, \quad n\ge 2\,.
$$
In the case $n=1$ this should be interpreted as $B_1=1$ (and is easily   proved directly). It can be checked by elementary means that the sequence $(B_n)_n$ is strictly increasing.
\par
The Dirichlet space $\cd $ is the set of all analytic functions in
the unit disc $\D$ with the finite area integral
$$
 \int_\D |f^\prime (z)|^2 d A(z) =  \sum_{n=1}^\infty n |b_n|^2\,.
$$
It is easy to see that $\cd \subset \cb$ but not the other way around. Thus, for a function $f\in\cb$, the area functional  $\sum_{k=1}^\infty k |b_k|^2$ need not be finite. However, any truncated area functional is bounded:
$$
 \F_n(f) = \sum_{k=1}^n k |b_k|^2 \le C_n \|f\|_\cb^2\,,
$$
where the constant $C_n$ depends only on $n$ but not on $f$. This has led Kayumov \cite{K} to show that, for small values of $n$, the above  truncated area functional is dominated by the supremum of its last term only:
\begin{equation}
 \sum_{k=1}^n k |b_k|^2 \le n B_n^2\,, \qquad 2\le n\le 5\,,
 \label{eqn-Kayumov}
\end{equation}
for all $\|f\|_\cb\le 1$ with $f(0)=0$, with equality if and only if $f$ is an appropriate constant multiple of $z^n$. Kayumov and Wirths \cite{KW1} then showed that the estimate also holds for $n=6$. Related extremal problems have also been considered recently in other papers (\textit{cf.\/} \cite{IK} and \cite{KW3}).
\par
A standard argument involving normal families can be used to show that, for each fixed $n\ge 1$, there exists an extremal function for the problem
\begin{equation}
 M_n = \sup \left\{ \F_n(f) = \sum_{k=1}^n k |b_k|^2 \,\colon\,\|f\|_\cb\le 1, f(z)=\sum_{k=1}^\infty b_k z^k \right\}\,.
 \label{eqn-Mn}
\end{equation}
Since $\|B_n z^n\|_\cb=1$, it is clear that $M_n \ge n B_n^2$. A natural question is whether estimate \eqref{eqn-Kayumov} extends for arbitrary $n\ge 2$. The present note is devoted to this question.
\par
We first display some evidence towards the correctness of the expected bound \eqref{eqn-Kayumov}. Specifically, as the main result  of Section~\ref{sect-pos-coeffs}, we prove this inequality for all $n$ for the functions with non-negative coefficients. This is the content of Theorem~\ref{thm-non-neg-coeffs}.
\par
Next, in Section~\ref{sect-weighted-est} we consider a more general weighted functional $\F^{t}_{n}(f) = \sum_{k=1}^{n} k^{t} |b_k|^{2}$ for $t\ge 0$ and the corresponding question as to whether $\F^t_n(f)\le n^t B_n^2$ for all functions $f$ with $\|f\|_\cb\le 1$ and $f(z)=\sum_{k=1}^\infty b_k z^k$. It is easy to see that this inequality holds for all $t\ge 2$. Other results obtained here imply that it is true for all $t\ge 1$ for the functions with non-negative Taylor coefficients. However, we show that, when $t<1$, this estimate fails for all sufficiently large $n$, even for functions with non-negative coefficients (see Theorem~\ref{thm-w-fcnl}). Thus, the exponent $t=1$ considered by Kayumov and Wirths is the critical value, which provides further motivation for studying generalizations of their results.
\par
In Section~\ref{sect-gen-res} of the paper we observe that  $\sum_{k=1}^n k |b_k|^2 \le \frac{n^n}{(n-1)^{n-1}}\le \frac{4}{e} n B_n^2$, for all $n\ge 2$ (Proposition~\ref{prop-est}). Although larger, this bound is asymptotically of the same order as $n B_n^2$. We also note that a typical Marty-type relation for the coefficients of extremal functions shows that at least one of the coefficients $b_n$ and $b_{n+1}$ must be zero, as expected. However, obtaining more detailed information from a number of possible variational methods has proved surprisingly difficult for us. Even proving that for any extremal function we must have $b_n\neq 0$ has eluded us so far.
\par
In the final Subsection~\ref{subsect-concl-rks} we include further discussion and remarks, showing that our Theorem~\ref{thm-non-neg-coeffs} does not imply directly the desired estimate in the general case. We include a relevant example in this respect.

\section{Functions with non-negative coefficients}
 \label{sect-pos-coeffs}
\par
In this section, we will prove that the bound obtained in \cite{K, KW1} is correct for all $n$ for the functions with non-negative Taylor coefficients.
\subsection{On extremal functions in the general setting}\label{subsect-extr-fcns}
The following observations will be useful. For any $f\in\cb$ and any constant $c$ we have $\F_n (cf)=|c|^2 \F_n (f)$. Thus, if $\|f\|_\cb<1$, since the function $f/\|f\|_\cb$ has unit norm and $\F_n (f/\|f\|_\cb)>\F_n(f)$, it is clear that $f$ cannot be an extremal function. In other words, if $F$ is an extremal function for \eqref{eqn-Mn}, we must have $\|F\|_\cb=1$.
\par
Extremal functions are clearly not unique since for a function $f\in\cb$ its rotation $f_\la$, given by $f_\la(z)=f(\la z)$,  $|\la|=1$, has the same norm and $\F_n(f_\la)=\F_n(f)$. Another issue is whether such a function is unique up to a rotation, which is what is expected by the findings from \cite{K} and \cite{KW1} for $2\le n\le 6$.
\subsection{Functions with non-negative coefficients}\label{subsect-pos-coeff}
We now turn to the restricted extremal problem for functions with non-negative coefficients:
\begin{equation}
 \sup \left\{ \F_n(f) = \sum_{k=1}^n k |b_k|^2 \,\colon\,\|f\|_\cb\le 1, \, f(z)=\sum_{k=1}^\infty b_k z^k, \, b_j\ge 0 \mathrm{\ for \ all \ } j\ge 1 \right\}\,.
 \label{eqn-Mn-restr}
\end{equation}
A few simple observations are in order. If $\Vert f \Vert_\cb\le 1$ and $f(z)=\sum_{j\ge  1} b_j z^j$ with $b_{j}\ge 0$ for all $j\ge 1$, it is obvious that $|f^\prime (z)| \le f^\prime (|z|)$, hence
$$
 \|f\|_\cb = \sup_{0\le r<1} (1-r^2) |f^\prime (r)|\,.
$$
Next, by the same argument as in Subsection~\ref{subsect-extr-fcns}, all extremal functions for the restricted problem \eqref{eqn-Mn-restr} have this property as well. If $f$ is such a function, it is clear that the truncated polynomial $p_n(z)=\sum_{j=1}^n b_j z^j$ (the Taylor polynomial of $f$ of order $n$) has the property that
$$
 \F_n (p_n) = \F_n (f) \quad \mbox{and} \quad \| p_n \|_\cb \le  \|f\|_\cb = 1\,.
$$
Moreover, if there exists at least one $k>n$ such that $b_k>0$, then clearly $\| p_n \|_\cb < 1$. But then $f$ cannot be extremal for $\F_n$ since $p_n/\|p_n\|_\cb$ and $f$ both have norm one and
$$
 \F_n \( \frac{p_n}{\|p_n\|_\cb} \) = \frac{1}{\|p_n\|_\cb^2} \F_n (f) > \F_n (f)\,.
$$
Thus, the only candidates for extremal functions for \eqref{eqn-Mn-restr} are polynomials of degree at most $n$. Note, however, that this consideration is not strictly necessary in the proof of the following lemma.
\begin{lem}\label{lem-coeff-est}
Let $f(z)=\sum_{j\ge  1} b_j z^j$ be an extremal function for the restricted extremal problem \eqref{eqn-Mn-restr}, with $b_j\ge 0$ for all $j\ge 1$. If $1\le k<m\le n$ and $b_k>0$, then
$$
 b_m\le \frac{m-k}{2m} b_k\,.
$$
\end{lem}
\begin{proof}
Let
$$
 g(z) = f(z) - b_k z^k + \frac{k}{m} b_k z^m\,.
$$
It is again an analytic function with non-negative coefficients. Since
$$
 g^\prime (r)= f^\prime (r) + k b_k (r^{m-1}-r^{k-1}) \le f^\prime (r)
$$
for $0\le r<1$, it follows that $\|g\|_\cb\le \|f\|_\cb$. Hence $g$ is in competition with $f$ and therefore
$$
 \F_n (g)= 2 k b_k b_m + \frac{k^2}{m} b_k^2 - k b_k^2 + \F_n(f) \le \F_n (f).
$$
Since $b_k>0$ by assumption, we obtain $2b_m + \frac{k}{m} b_k  \le b_k$, and the desired conclusion follows.
\end{proof}
\par
We are now ready to prove the main result of this section.
\par
\begin{thm}\label{thm-non-neg-coeffs}
If $\| f \|_\cb \le 1$ and $f(z)=\sum_{j\ge  1} b_j z^j$ with $b_j\ge 0$ for all $j\ge 1$, then
\[
 \F_n(f)=\sum_{k=1}^{n} k |b_k|^{2} \le n B_n^2, \quad \mathrm{for \ all \ } n\ge 2\,.
\]
Moreover, equality is only achieved for $f(z)= B_{n}z^{n}$.
\end{thm}
\begin{proof}
First of all, we observe that it suffices to prove the statement only for extremal functions. We proceed by induction on $n$. As is well known, the statement is true for $n=2$. Let $n\ge 3$ and assume the inequality is true for all $k<n$. We will now show that it also holds for $n$.
\par
If $f$ is extremal for \eqref{eqn-Mn-restr}, as observed before, $f$ is a polynomial of degree at most $n$. Suppose that $b_j>0$ for some $j$ with $1\le j<n$, and let  $k = \max\{j\,:\,1\le j<n \mathrm{\ and \ }  b_j>0\}$. Recalling that the sequence $(B_n)_n$ is increasing, by our inductive assumption, we clearly have
$$
 n B_k^2 \le n B_n^2 \le \F_n(f) = \sum_{j=1}^k j b_j^2 + n b_n^2 = \F_k(f) + n b_n^2 \le k B_k^2 + n b_n^2\,.
$$
It follows that $n b_n^2 \ge (n-k) B_k^2$ and therefore
$$
 b_n \ge \sqrt{\frac{n-k}{n}} B_k > \frac{1}{2} \frac{n-k}{n} B_k \ge \frac{1}{2} \frac{n-k}{n} b_k \ge b_n
$$
in view of Lemma~\ref{lem-coeff-est} (with $m=n$), which is clearly absurd. Thus, it follows that $b_j=0$ for all $j$ with $1\le j<n$. As was already observed, only functions of norm one can be extremal and it follows immediately that  $B_n z^n$ is the only possible extremal function in this case.
\end{proof}
\par
It should be noted that Theorem~\ref{thm-non-neg-coeffs} is also true for functions $f$ with $f(z)=\sum_{k\ge 1} b_{k} z^{k}$ with $b_{k}=|b_{k}|e^{i(a+k b)}$, for some $a$, $b \in \mathbb{R}$. This is clear since, given a function $f$ of this kind, the function $e^{-ia}f\left( e^{-i b} z \right)$ also vanishes at the origin and has the same norm as $f$ and non-negative coefficients.

\section{On more general weighted estimates}
 \label{sect-weighted-est}
\par
As was noted in \cite[p.~467]{K}, it is quite simple to show that for any function with $\| f \|_\cb\le 1$ and $f(z)=\sum_{k \ge 1} b_k z^{k}$ the sharp estimate
$$
 \sum_{k=1}^{n} k^{2} |b_k|^{2}\le n^{2} B_n^{2}
$$
holds for all $n\ge 1$. This suggests the idea of studying the generalized functional
$$
 \F^{t}_{n}(f)= \sum_{k=1}^{n} k^{t} |b_k|^{2}
$$
with $t\ge 0$. A natural question is whether $\F^{t}_{n}(f)\le n^{t} B_n^{2}$, under the same hypotheses as above, and whether the inequality is sharp. We already know that this is true when $t=2$ and the problem studied previously constitutes the case $t=1$.
\par
It is easy to see that if the above weighted estimate is valid for all $n$ for the functional $\F^t_{n}$, then the analogous estimate is also satisfied for the functional $\F^{s}_{n}$, whenever $s\ge t$. Indeed, if $\| f \|_\cb\le 1$, $f(z)=\sum_{k \ge 1} b_k z^{k}$, the inequality
\[
 \F^{t}_{n}(f)= \sum_{k=1}^{n} k^{t} |b_k|^{2}\le n^{t} B_n^{2}
 \]
is satisfied and $s\ge t$, then clearly
\[
 \F^{s}_{n}(f) = \sum_{k=1}^{n} k^{s} |b_k|^{2} \le n^{s-t}\sum_{k=1}^{n} k^{t} |b_k|^{2} \le  n^{s-t} n^{t} B_n^{2} = n^{s}B_n^{2}.
\]
Thus, for all functions $f$ such that $\| f \|_\cb\le 1$ and $f(z)=\sum_{k \ge 1} b_k z^{k}$, it follows that
\begin{equation}
 \sum_{k=1}^{n} k^t |b_k|^2 \le n^t B_n^{2},
 \label{eqn-gen-est}
\end{equation}
whenever $t\ge 2$. Additionally, if $f$ has non-negative Taylor coefficients, inequality \eqref{eqn-gen-est} actually holds whenever $t\ge 1$ in view of Theorem~\ref{thm-non-neg-coeffs}.
\par
The following result will show that for all $t$ with $t<1$ the above estimate \eqref{eqn-gen-est} fails in a strong way. In fact, it cannot even be true for the functions with non-negative Taylor coefficients when $n$ is sufficiently large. This shows that the exponent $t=1$ (the basic case studied here and in \cite{K, KW1}) is really critical.
\begin{thm}\label{thm-w-fcnl}
Given $t<1$, there exist a positive integer $N$ that depends only on $t$ and such that for all $n\ge N$ there exists a function $f_n\in \mathcal{B}$ with non-negative Taylor coefficients, $f_n(0)=0$, $\|f_n\|_\cb\le 1$, and with the property that
\[
 \F^{t}_{n}(f_n) > n^{t} B_n^{2}\,.
\]
\end{thm}
\begin{proof}
Let $\varepsilon >0$ such that $t+\varepsilon<1$, for example, $\e=(1-t)/2$. Choose an integer $N\ge (2e^2)^{2/(1-t)}$ and, for $n\ge N$, consider the functions
\[
 f_n (z) = z + b_n z^{n}, \quad b_n = \sqrt{B_n^{2} - \frac{1}{n^{t+\varepsilon}}}.
\]
Clearly,
\[
 \F^{t}_{n}(f_n) = 1+n^{t}\left(B_n^{2}- \frac{1}{n^{t+\varepsilon}} \right) > n^{t}B_{n}^{2},
\]
hence it is only left to check that $\|f_n\|_\cb\le 1$ for all $n\ge N$. To this end, note that
\[
 \|f_n\|_\cb=\sup_{z\in \mathbb{D}}(1- |z|^{2}) | 1+n b_n z^{n-1}| = \sup_{0\le r <1}(1- r^{2}) \left( 1+n b_n r^{n-1}\right).
\]
Thus, in order to finish the proof, we need only show that for $n\ge N$ and for all $r\in (0,1)$ the inequality $(1- r^{2}) \left( 1+n b_n r^{n-1}\right)\le 1$ holds. The  last statement is easily seen to be equivalent to
$$
 n b_n r^{n-3}(1- r^{2}) \le 1, \quad \mathrm{for \ all \ } r\in (0,1) \quad \mathrm{and \ } n\ge N\,.
$$
If we define $h(r)=n b_n r^{n-3}(1- r^{2})$, it is clear that
\[
 \max_{0\le r \le 1} h(r)= \max \left\{ h(0), h\left( \sqrt{\frac{n-3}{n-1}} \right), h(1)\right\} =  n b_n \left( \frac{n-3}{n-1} \right)^{\frac{n-3}{2}} \frac{2}{n-1}.
\]
It is, thus, only left to show that $n b_n \left( \frac{n-3}{n-1} \right)^{\frac{n-3}{2}} \frac{2}{n-1} \le 1$ for $n\ge N$. The following chain of equivalences is clear:
\begin{align*}
 n b_n \left( \frac{n-3}{n-1} \right)^{\frac{n-3}{2}} \frac{2}{n-1} \le 1  & \Longleftrightarrow  b_n^2 \le \frac{(n-1)^{n-1}}{4n^{2} (n-3)^{n-3}} \Longleftrightarrow B_n^{2}- \frac{1}{n^{t+\varepsilon}} \le \frac{(n-1)^{n-1}}{4n^{2} (n-3)^{n-3}}
\\
 & \Longleftrightarrow   \frac{ (n+1)^{n+1}}{4 n^{2} (n-1)^{n-1}}- \frac{(n-1)^{n-1}}{4n^{2} (n-3)^{n-3}}\le \frac{1}{n^{t+\varepsilon}}\\
 &  \Longleftrightarrow   \frac{(n+1)^2 \left( 1+ \frac{ 2}{ n-1}\right)^{n-1 }- (n-1)^2 \left( 1+ \frac{ 2}{ n-3}\right)^{n-3 }}{n}\le 4n^{1-t-\varepsilon}\,.
\end{align*}
Obviously, the right-hand side  in the last inequality tends to infinity when $n$ tends to infinity. On the other hand, using the classical inequalities $ \left( 1+\frac{1}{x}\right)^{x}< e <\left( 1+\frac{1}{x}\right)^{x+1} $ for all $x>0$, we see that the left-hand side in the last inequality in the chain above satisfies
\begin{align*}
 \frac{(n+1)^2 \left( 1+ \frac{ 2}{ n-1}\right)^{n-1 }- (n-1)^2 \left( 1+ \frac{ 2}{ n-3}\right)^{n-3 }}{n}& \le e^2  \frac{(n+1)^2 - (n-1)^2 \left( 1+ \frac{ 2}{ n-3}\right)^{-2}}{n}\\
 &=e^2  \frac{(n+1)^2 - (n-3)^2}{n}\\
 &=e^2  \frac{8n-8}{n}\\
 &\le 8e^{2}\,.
\end{align*}
Since $N \ge (2e^2)^{2/(1-t)}$ and $\e=(1-t)/2$ (by our choices), it is clear that $4n^{1-t-\varepsilon}\ge 8e^{2}$ and the desired conclusion follows.
\end{proof}

\section{Discussion of the general case}
 \label{sect-gen-res}
\par
\subsection{An estimate of correct asymptotic order}
 \label{subsect-crude-est}
If the conjectured bound on the truncated functional is true for all $n$, the following estimate cannot be sharp. However, it may still be useful so we mention it here.
\begin{prop} \label{prop-est}
If $\|f\|_\cb\le 1$ and $f(z)=\sum_{k=1}^\infty b_k z^k$, then
\begin{equation}
 \sum_{k=1}^n k |b_k|^2 \le \frac{n^n}{(n-1)^{n-1}}\,, \quad \mathrm{for \ all \ } n\ge 2\,.
 \label{eqn-estimate}
\end{equation}
This estimate is of the same asymptotic order as the conjectured estimate since
$$
 \frac{32}{27} n B_n^2 \le \frac{n^n}{(n-1)^{n-1}} \le \frac{4}{e} n B_n^2\,, \quad \mathrm{for \ all \ } n\ge 2\,.
$$
\end{prop}
\begin{proof}
Starting with the standard estimate based on Parseval's equality \cite{ACP, K}:
\begin{equation}
 \sum_{k=1}^n k^2 |b_k|^2 \r^{k-1} \le \frac{1}{(1-\r)^2}\,, \qquad 0\le \r<1\,,
 \label{eqn-basic-fla}
\end{equation}
and integrating with respect to $\rho$ from $0$ to $x$, we obtain:
$$
 \sum_{k=1}^n k |b_k|^2 x^{k-1} \le \frac{1}{1-x}\,, \qquad 0\le x<1\,.
$$
Since $x^{k-1} \ge x^{n-1}$ whenever $1\le k\le n$, it follows that
$$
 \sum_{k=1}^n k |b_k|^2 \le \frac{1}{x^{n-1} (1-x)}
$$
for all $x\in (0,1)$. The right-hand side attains its minimum value
\begin{equation}
 \frac{n^n}{(n-1)^{n-1}}
 \label{eqn-bound}
\end{equation}
at the point $x=(n-1)/n$. This yields the desired estimate \eqref{eqn-estimate}.
\par
As for the order of the estimate obtained, this is easily seen by elementary calculus since the sequence
$$
 \frac{n^n}{(n-1)^{n-1}}\bigg/ (n B_n^2) = \frac{4}{\(\frac{n+1}{n}\)^{n+1}}
$$
is increasing and convergent to $\frac{4}{e} \simeq   1.4715$.
\end{proof}
\par\medskip
\subsection{Remarks on the $n$-th coefficient of an extremal function}\label{subsect-coeff-extr-fcn}
It is desirable to obtain further partial information on the Taylor coefficients of extremal functions for the functional $\F_n$. For example, one would like to show that we must always have $b_n\neq 0$. However, this seemingly simple point has eluded us so far and we have not been able to find a satisfactory proof.
\par
It is, however, relatively simple to show that for every extremal function either $b_{n+1}=0$ or $b_n=0$ by working out an analogue of the Marty variation as in the theory of univalent functions (\textit{cf.\/} \cite[pp.~59--60]{D}). This can be proved by using conformal invariance; more specifically, if $F$ is extremal, the function
$$
 F_\la(z) = F\( \frac{z+\la}{1+\overline{\la}z} \) - F(\la) = \sum_{k=1}^\infty B_k(\la) z^k
$$
has Bloch norm one and satisfies $F_\la(z)=0$, hence it competes with $f$: $\sum_{k=1}^n k |B_k(\la)|^2 \le \sum_{k=1}^n k |b_k|^2$. By considering the asymptotic expansion:
$$
 |B_k(\la)|^2 = |b_k|^2 + 2\mathrm{Re}\,\left\{ \( (k+1) b_{k+1} \overline{b_k} - (k-1) \overline{b_{k-1}} b_k \) \la \right\} + O(|\la|^2)\,, \qquad k\ge 1\,,
$$
for $\la$ close to zero, where $b_0=0$, after some algebra, this leads to
$$
 \mathrm{Re}\, \left\{ \( n (n+1) b_{n+1} \overline{b_n}\) \la\right\} + O(|\la|^2) \le 0
$$
and, upon dividing by $|\la|$ and letting $|\la|\to 0$, since the argument of $\la$ can be arbitrary, we conclude that $n (n+1) b_{n+1} \overline{b_n} = 0$, which proves the statement.
\par\smallskip
As a way of restating the desired property of the $n$-th coefficients of extremal functions for problem \eqref{eqn-Mn}, it is not difficult to show that the sequence $(M_n)_n$ of numbers defined by \eqref{eqn-Mn} is strictly increasing if and only if for each extremal function $F(z) = \sum_{k=1}^\infty b_k z^k$ we have $b_n\neq 0$.
\par
\subsection{Concluding remarks}\label{subsect-concl-rks}
We end this note by pointing out some limitations of the method employed in Subsection~\ref{subsect-pos-coeff}.
\par
Regarding Theorem~\ref{thm-non-neg-coeffs}, one may ask whether in the general case we may restrict our attention to polynomials only. However, knowing only that $\F_n(p)\le n B_n^2$ for all polynomials $p$ of unit norm does not imply in a direct way the same inequality for general functions of norm one. This can be seen as follows. While for a function $f$ in the unit ball and its $n$-th Taylor polynomial $p_n$ we have $\F_n(f)=\F_n(p_n)$, it can be seen that there are functions $f$ in the unit ball for which $\|p_n\|_\cb >1$, hence for the normalized function $p_n/\|p_n\|_\cb$ we obtain $\F_n(p_n/\|p_n\|_\cb) = \F_n(f)/\|p_n\|_\cb^2 < \F_n(f)$, showing that the maximum of the functional over the normalized polynomials could in theory be strictly smaller than the global maximum for problem \eqref{eqn-Mn}. Of course, showing that the conjectured bound is correct would amount to proving that such functions $f$ cannot be extremal; hence, some additional work is required.
\par
Another natural question is whether from the estimate obtained only for functions with non-negative coefficients in Theorem~\ref{thm-non-neg-coeffs} we can immediately deduce the result for arbitrary functions. However, this does not seem so simple since,  employing the usual terminology, the unit ball of $\cb$ is not \textit{solid\/}. In other words, if $f(z)=\sum_{k=1}^\infty b_k z^k$ is in the ball of $\cb$, the function $F(z)=\sum_{k=1}^\infty |b_k| z^k$ need not belong to the ball of $\cb$. It is clear that $F$ is also analytic in the disc and $\|f\|_\cb \le \|F\|_\cb$ by crude estimates based on the triangle inequality, but there are examples showing that this inequality can actually be strict. More generally speaking, a change of signs in some of the coefficients of a function could enlarge its norm. Thus, if $f$ and $F$ are as above (hence $F$ has non-negative coefficients) and also $1=\|F\|_\cb > \|f\|_\cb$, then $f/\|f\|_\cb$ is also of unit norm and
$$
 \F_n\(\frac{f}{\|f\|_\cb}\) = \frac{\F_n(F)}{\|f\|_\cb^2} > \F_n(F)\,, $$
so $F$ cannot be extremal. Hence, a proof of the desired estimate in the general case would show in particular that functions $F$ with the above property cannot be extremal.
\par
One single example can be used to illustrate both phenomena (after the appropriate normalization).
\begin{ex} \label{ex-norms}
Let $n\ge 2$ and let $\e$ be sufficiently small, say $0<\e\le 1/5$. Consider the function $f(z)=z+B_n z^n - \frac{\e z^{2n-1}}{2n-1}$. The associated functions discussed above are $F(z)=z+B_n z^n + \frac{\e z^{2n-1}}{2n-1}$ and $p_n(z)=z+B_n z^n$ and we have the inequality
\begin{equation}
 \norm{ F }_\cb > \norm{ p_n}_\cb > \norm{ f }_\cb\,.
 \label{eqn-ineq-norms}
\end{equation}
\end{ex}
To check the first inequality in \eqref{eqn-ineq-norms}, recall from Subsection~\ref{subsect-pos-coeff} that for every function with positive coefficients in $\cb$ the supremum defining its norm is attained on the radius $[0,1)$. In particular, this shows that
$$
 \norm{F}_\cb = \sup_{0\le r<1} (1-r^2) (1+ n B_n r^{n-1} + \e r^{2n-2}) > \sup_{0\le r<1} (1-r^2) (1+ n B_n r^{n-1}) = \norm{p_n}_\cb \,.
$$
The above inequality holds because we have the strict inequality
$$
 (1-r^2) (1+ n B_n r^{n-1} + \e r^{2n-2}) > (1-r^2) (1+ n B_n r^{n-1})\,, \quad \mathrm{ for } \ 0<r<1\,,
$$
equality holds for $r=0$ and at $r=1$, and also because the smaller quantity is strictly bigger than $1$ (the value of both sides at $r=0$) for at least one $r\in (0,1)$. This last observation follows in view of the fact that $\|B_n z^n\|_\cb=1$ and this norm is attained at the value $r_n=\sqrt{\frac{n-1}{n+1}}\in (0,1)$, hence
$$
 \norm{p_n}_\cb = \sup_{0\le r<1} (1-r^2) (1+ n B_n r^{n-1}) \ge (1-r_n^2) (1 + n B_n r_n^{n-1}) > (1-r_n^2) n B_n r_n^{n-1} = 1\,.
$$
\par
The second inequality in \eqref{eqn-ineq-norms} requires more work. In order to determine $\norm{z+B_n z^n - \frac{\e z^{2n-1}}{2n-1}}_\cb$, we need to obtain an upper estimate on $|1+ n B_n z^{n-1} - \e z^{2n-2}|$. Writing $z^{n-1}=R e^{i t}$ and then $R=r^{n-1}$, $\cos t=x$, one easily computes
\begin{eqnarray*}
 |1+ n B_n z^{n-1} - \e z^{2n-2}|^2 &=& |1+ n B_n R e^{i t} - \e R^2
 e^{2 i t}|^2
\\
 &=& 1 + n^2 B_n^2 R^2 + \e^2 R^4 + 2n B_n R x - 2 n B_n \e R^3 x - 2 \e R^2 (2x^2-1)
\\
 &=& - 4 \e R^2 x^2 + 2n B_n R (1 - \e R^2) x + 1 + n^2 B_n^2 R^2 + 2 \e R^2 + \e^2 R^4
\\
 &=& u(x) \,.
\end{eqnarray*}
Then $u^\prime (x) = 2 R \( n B_n (1 - \e R^2) - 4 \e R x\) \ge 0$ in view of the obvious inequalities $n B_n (1 - \e R^2) \ge 1 - \e R^2  \ge 1- \e\ge 4 \e$. Thus, $u$ attains its maximum at $x=1$ and the maximum value is
$$
 u(1) = 2n B_n R (1 - \e R^2) + 1 + n^2 B_n^2 R^2 - 2 \e R^2 + \e^2 R^4 = (1 + n B_n R)^2 + \e R^2 (\e R^2 - 2 - 2n B_n R)\,.
$$
The second summand in the last expression is negative, hence $u(x) \le u(1) < (1 + n B_n R)^2$. This yields
$$
 \norm{z+B_n z^n - \frac{\e z^{2n-1}}{2n-1}}_\cb \le \sup_{0\le r<1} (1-r^2) (1 + n B_n r^{n-1}) = \norm{p_n}_\cb\,.
$$
Again, the inequality is easily seen to be strict since $\|p_n\|>1$, hence this norm is attained for some $r\in (0,1)$, while the inequality in $u(x) < (1 + n B_n R)^2 = (1 + n B_n r^{n-1})^2$ is strict for $r\in (0,1)$.



\begin{thebibliography}{CM}

\bibitem{ACP}
J.M. Anderson, J. Clunie, Ch. Pommerenke, On Bloch functions and normal functions, \textit{J. Reine Angew. Math.\/} \textbf{270} (1974), 12–-37.

\bibitem{D}
P.L. Duren, \textit{Univalent Functions\/}, Springer-Verlag, New~York 1983.

\bibitem{IK}
O.V. Ivri\u{\i}, I.R. Kayumov, Makarov's principle for the unit ball in Bloch space, \textit{Mat. Sb.\/} \textbf{208} (2017), no. 3, 96–-110; translation in \textit{Sb. Math.\/} \textbf{208} (2017), no. 3-4, 399–-412.

\bibitem{K}
I.R. Kayumov, A note on an area-type functional of Bloch functions,
\textit{Lobachevskii J. Math.\/} \textbf{38} (2017), no. 3, 466–-468.

\bibitem{KW1}
I.R. Kayumov, K.-J. Wirths, Coefficient inequalities for Bloch functions, \textit{Lobachevskii J. Math.\/} \textbf{40} (2019), No. 9, pp. 1319--1323.

\bibitem{KW2}
I.R. Kayumov, K.-J. Wirths, On the sum of squares of the coefficients of Bloch functions, \textit{Monatsh. Math.\/} \textbf{190} (2019), no. 1, 123–-135.

\bibitem{KW3}
I.R. Kayumov, K.-J. Wirths, Coefficients problems for Bloch functions, \textit{Anal. Math. Physics\/} \textbf{9} (2019), 1069–-1085.

\bibitem{W}
K.-J. Wirths, Über holomorphe Funktionen, die einer Wachstumsbeschränkung unterliegen (German), \textit{Arch. Math. (Basel)\/} \textbf{30} (1978), no. 6, 606–-612.

\end{thebibliography}
\end{document}